\documentclass[11pt,oneside]{amsart}

\usepackage{bbm}
\usepackage{mathrsfs}
\usepackage{soul}

\usepackage[english]{babel}
\usepackage[utf8]{inputenc}
\usepackage[T1]{fontenc}

\usepackage[top=3cm,bottom=2cm,left=3cm,right=3cm,marginparwidth=1.75cm]{geometry}
\usepackage{physics}
\usepackage{amsmath,amssymb}
\usepackage{mathtools} 
\mathtoolsset{
showmanualtags}
\usepackage{esint} 
\usepackage{graphicx}
\usepackage[dvipsnames]{xcolor}
\usepackage[colorlinks=true,allcolors=BlueViolet, pagebackref=true]{hyperref}
\renewcommand*\backref[1]{\ifx#1\relax \else (Cited on p.#1) \fi}
\hypersetup{hypertexnames=false}
\usepackage{cleveref}
\usepackage[shortlabels]{enumitem}

\usepackage{dsfont} 

\usepackage{xspace} 
\usepackage{cancel}
\usepackage{fancyvrb}

\DeclareFontFamily{U}{BOONDOX-calo}{\skewchar\font=45 }
\DeclareFontShape{U}{BOONDOX-calo}{m}{n}{
  <-> s*[1.05] BOONDOX-r-calo}{}
\DeclareFontShape{U}{BOONDOX-calo}{b}{n}{
  <-> s*[1.05] BOONDOX-b-calo}{}
\DeclareMathAlphabet{\mathbdx}{U}{BOONDOX-calo}{m}{n}
\SetMathAlphabet{\mathbdx}{bold}{U}{BOONDOX-calo}{b}{n}
\DeclareMathAlphabet{\mathbbdx}{U}{BOONDOX-calo}{b}{n}

\DeclareMathOperator{\arcsinh}{arcsinh}

\newtheorem{theorem}{Theorem}[section]
\newtheorem{proposition}[theorem]{Proposition}

\newtheorem{lemma}[theorem]{Lemma}

\newtheorem{definition}[theorem]{Definition}

\newtheorem{manualtheoreminner}{Theorem}
\newenvironment{manualtheorem}[1]{%
  \IfBlankTF{#1}
    {\renewcommand{\themanualtheoreminner}{\unskip}}
    {\renewcommand\themanualtheoreminner{#1}}%
  \manualtheoreminner
}{\endmanualtheoreminner}
  
\theoremstyle{remark} 

\newtheorem{remark}[theorem]{Remark}

\crefname{equation}{Equation}{Equations}
\crefname{gather}{Equation}{Equations}
\crefname{multline}{Equation}{Equations}
\crefname{figure}{Figure}{Figures}
\crefname{question}{Question}{Question}
\crefname{section}{Section}{Sections}
\crefname{subsection}{Subsection}{Subsections}
\crefname{appendix}{Appendix}{Appendices}
\crefname{lemma}{Lemma}{Lemmas}
\crefname{proposition}{Proposition}{Propositions}
\crefname{theorem}{Theorem}{Theorems}
\crefname{innercustomthm}{Theorem}{Theorems}
\crefname{mainthm}{Theorem}{Theorems}
\crefname{corollary}{Corollary}{Corollaries}
\crefname{definition}{Definition}{Definitions}
\crefname{remark}{Remark}{Remarks}
\crefname{proposition}{Proposition}{Proposition}
\crefname{corollary}{Corollary}{Corollaries}
\crefname{example}{Example}{Examples}
\crefname{claim}{Claim}{Claim}
\crefname{conjecture}{Conjecture}{Conjecture}

\definecolor{bluola}{RGB}{138,43,226}

\newcommand{\R}{\mathbb{R}}
\newcommand{\C}{\mathbb{C}}
\newcommand{\N}{\mathbb{N}}
\newcommand{\Z}{\mathbb{Z}}
\renewcommand{\P}{\mathbb{P}}
\newcommand{\E}{\mathbb{E}}
\newcommand{\F}{\mathcal{F}}

\newcommand{\de}{\partial}

\newcommand{\inter}[1]{%
  {\kern0pt#1}^{\mathrm{o}}%
}

\newcommand{\1}{\mathds{1}}
\newcommand{\f}{\varphi}
\renewcommand{\a}{\alpha}
\renewcommand{\b}{\beta}

\newcommand{\h}{\theta}
\newcommand{\zero}{0}

\newcommand{\vol}[1]{\mathrm{Vol}^{#1}}

\newcommand{\Var}{{\mathbb{V}\mathrm{ar}}}
\newcommand{\Cov}{\mathrm{Cov}}
\newcommand{\coeff}{\Theta}

\newcommand{\lft}{\mathcal{L}_{f-t}}

\def\randin{%
  \mathchoice%
    {\raisebox{-.35ex}{$\displaystyle{^\subset}$}\mkern-11.5mu\raisebox{+.45ex}{$\displaystyle{_\subset}$}}
    {\mkern+1mu\raisebox{-.27ex}{$\textstyle{^\subset}$}\mkern-11.7mu\raisebox{+.45ex}{$\textstyle{_\subset}$}}
    {\raisebox{.35ex}{$\scriptstyle\subset$}\mkern-14mu\raisebox{-.15ex}{$\scriptstyle\subset$}}
    {\raisebox{.3ex}{$\scriptscriptstyle\subset$}\mkern-13.5mu\raisebox{-.10ex}{$\scriptscriptstyle\subset$}}
}

\makeatletter
\newcommand{\tpitchfork}{%
  \raise-0.1ex\vbox{
    \baselineskip\z@skip
    \lineskip-.52ex
    \lineskiplimit\maxdimen
    \m@th
    \ialign{##\crcr\hidewidth\smash{$-$}\hidewidth\crcr$\pitchfork$\crcr}
  }%
}

\newcommand{\mC}{\mathcal{C}}
\newcommand{\m}[1]{\mathcal{#1}}
\newcommand{\be}{\begin{equation}}
\newcommand{\ee}{\end{equation}}
\numberwithin{equation}{section}

\newcommand{\bega}{\begin{equation}\begin{aligned}}
\newcommand{\eega}{\end{aligned}\end{equation}}

\newcommand{\begt}{\begin{equation}\begin{gathered}}
\newcommand{\eegt}{\end{gathered}\end{equation}}
\newcommand{\kop}{\left\{}
\newcommand{\pok}{\right\}}
\newcommand{\tyu}{\left(}
\newcommand{\uyt}{\right)}
\newcommand{\qwe}{\left[}
\newcommand{\ewq}{\right]}
\DeclareMathOperator{\Ho}{Ho}

\newcommand{\dist}[2]{\mathrm{dist}\!\tyu #1 , #2\uyt}

\newcommand{\spin}[1]{\m{T}^{\otimes {#1}}}
\newcommand{\Lag}{\Sigma}
\newcommand{\immondo}{\kappa}

\title[Level area of spin random fields]{Level area of spin random fields: a chaos decomposition}
\author{Francesca Pistolato, Michele Stecconi}
\date{\today}
\begin{document}

\begin{abstract}
We study real left-invariant spin Gaussian fields on $SO(3)$, a special class of non-isotropic random fields used to model the polarization of the Cosmic Microwave Background.
Leveraging recent results from \emph{New chaos decomposition of Gaussian nodal volumes} \cite{cgv2025StecconiTodino}, we provide an explicit formula for the Wiener-It\^o chaos decomposition of the area measure of level sets of such random fields. Our analysis represents a step forward in the study of second-order asymptotic properties of the Lipschitz-Killing curvatures of excursion sets of spin random fields. Remarkably, our formulas reveal a clear difference between the high frequency regime and the zero spin case. 
\end{abstract}

\maketitle 

\section{Introduction} 
Since its discovery in 1965 \cite{1965_penzias_wilson}, cosmologists have shown increasing interest in the Cosmic Microwave Background (CMB), a primordial {radiation field} 
carrying significant information about the early stages of the universe. In particular, it is one of the main pieces of evidence supporting the current cosmological model \cite{AHeavens_2008}; more and more precise constraints on cosmological parameters can be extracted from the observed angular variation of its temperature (or intensity) correlations, and polarization, see \cite{1992_cobe,2013_wmap,Planck20}. We refer to \cite[Chapter 29]{2022_reviewCMB} for a complete yet concise review of CMB literature and perspectives. 

The latest research direction 
{on this topic} concerns the study of CMB polarization \cite{geller2008spin,Cabella:2004mk,Seljak}. According to inflation theory, rapid expansion in the early universe generated primordial gravitational waves, which would {imprint} a distinctive swirling {polarization pattern on the CMB radiation} (B-mode) {superposed on the simpler scalar fluctuations (E-mode)}. Precise measurements of {the} CMB polarization would then help in differentiating between competing inflationary models, 
{determine the energy level at which inflation occurred} and provide insights into the fundamental mechanisms that drove the universe's initial expansion. The forthcoming space mission LiteBIRD is designed to map the polarization of CMB across the entire sky with unprecedented sensitivity, see \cite{LiteBIRD_Collaboration2023-jy,CMB2023arXiv231200717C}. 

A rich direction of research aims to detect deviations from Gaussianity and isotropy in CMB polarization data. To this end, it is of outmost importance to devise statistics sensitive to these features. Whilst the study of the angular power spectrum has led to tightest constraints in the values of cosmological parameters \cite{Planck20}, it 
{does not allow to} differentiate between Gaussian and non-Gaussian models, 
{and it is} blind to possible anisotropies. Complementary tools in this analysis are the Lipschitz-Killing curvatures, better known as Minkowski functionals (MFs) in the physics literature \cite{Schmalzing1997}, of the excursion set. They are a set of descriptors for the {local} geometry and topology of such fields (in dimension $3$, they are: volume, surface area, total mean curvature, and Euler characteristic) and prove to be sensitive to higher-order correlations. 

In \cite{articolo_dei_fisici_published}, the study of Minkowski functionals has been extended to non-scalar fields such as the CMB polarization, 
modeled as a spin-$2$ field on the $2$-sphere, by relating such models to a special kind of real-valued scalar fields on the orthogonal group $SO(3)$, {see also \cite{BR13}}. In \cite{articolo_dei_fisici_published}, simulations prove that (conjectured yet probed) theoretical predictors are fully 
{consistent} with the MFs computed on CMB Gaussian isotropic maps with a realistic angular power
spectrum, simulated according to \cite{Planck2020_V}. Indeed, they find no systematic deviations in the residuals, as well as a low statistical variation of the curves, which decreases when increasing the resolution ({high frequency} limit). In \cite{elk2024PistolatoStecconi}, the authors have confirmed the conjectures and derived the exact formulas of the predictors. 

In the present work, we move a small step forward addressing the study of second-order results for such functionals. Indeed, the chaos decomposition proved to be a very efficient tool to study the variance and further fluctuations of nodal volume, and other geometric functionals, of Gaussian fields; the following list of representative references is by no means complete: \cite{MRossiWigman2020,MPRW,kratzleon,Cammarota2017NodalAD,Not23,Smutek,NourdinPeccatiRossi2019,maini2025,MAINI2024,letendre_varvolII_2019,VidLipschitz,2021_3dBerry_Dalmao_Estrad_Leon,BM19,marinucci2023laguerre,Gass2025}, see also the surveys \cite{rossisurvey,Wigman2022}. 
All past works concern isotropic settings, where the field is \emph{homothetic}, but this is not the case for spin fields, as shown in \cite{stecconi2021isotropic}. We discuss this point more in detail in \cref{sec:nohomo} below.

Our main contribution is a chaos expansion for the area of the level set of a spin-$2$ field, in the formalism used by \cite{articolo_dei_fisici_published}, that is the second Lispchitz-Killing curvature of the excursion set, measured with respect to the standard geometry of $SO(3)$. 
Furthermore, we present our results within a unified framework for both the level area of spin ($s\in\Z$) random fields, and boundary length of random spherical ($s=0$) harmonics. This allows us to directly compare the two decompositions and distinguish between the two cases ($s\neq 0$ and $s=0$) by examining higher-order terms in the {chaos} expansion in the high frequency regime. This is a novelty, as no distinction can be detected at first order (see \cite{elk2024PistolatoStecconi}), or looking at the scaling limit (see \cite{geospin2022LMRStec}). 
\subsection{Notations}\label{sec:notations}

\begin{enumerate}[(i)]

\item {For the rest of the paper, and unless otherwise specified, every random element is assumed to be defined on an adequate common probability space $(\Omega,\mathcal F, \P)$, with $\E$ denoting the expectation with respect to $\P$.}

\item A \emph{random element} (see \cite{Billingsley}) of the topological space $T$ (or \emph{with values} in $T$) is a measurable mapping $X\colon \Omega\to T$, defined on a probability space $\tyu \Omega,\mathscr{E},\P \uyt$. In this case, we write
    \be\label{eq:randin}
    X\randin T .
    \ee 

\item We will denote the $n$-volume of the 
unit sphere $S^{n}=\kop x \in \R^{n+1}\colon |x|=1\pok$ as $s_{n}=\frac{2\pi^{\frac{n+1}2}}{\Gamma\tyu\frac{n+1}{2}\uyt}$.

\end{enumerate}

\subsection{Plan of the paper} 
In \cref{sec:setting} we present our object of interest, \cref{eq:length}, and state our two main contributions: \cref{thm:3} and \cref{prop:s0}, the latter allowing for a comparison with analogous results in the spherical case. In \cref{sec:propf} and \cref{sec:propgf} we present known and new properties of the field and its induced geometry. \cref{sec:chaos} is devoted to preliminaries and proofs. To conclude, in \cref{appendix} we present some results on Gaussian random variables that we believe to be of independent interest.

\section{Setting and main result} \label{sec:setting}

\subsection{The field}
We consider the family of Gaussian random fields $f$ on $SO(3)$, defined as in \cite[Eq. (1.2)]{elk2024PistolatoStecconi}:
\begin{equation}\label{eq:f}
f=\Re(X),\quad \text{where} \quad X= \sum_{l=|s|}^{\infty}c_l\sum_{m=-l}^l \gamma^l_{m,s} D^l_{m,s},
\end{equation}
where $s\in \Z$, $l,m\in \N$; $\gamma^l_{m,s}$ {being standard symmetric complex Gaussian variable $\zeta \sim \mathcal N_{\mathbb C}(0,1)$ (meaning $\Re \gamma ^l_{ms},\Im \gamma ^l_{ms}\sim \m N(0,\frac{1}{2})$ are independent), see also \cite[Proposition 1.31]{janson}}; the deterministic functions $D^l_{m,s}:SO(3)\to \C$ are the coefficients of Wigner matrices (see \cite[Section 3.3]{libro}), and the numbers $c_l>0$ are real positive constants normalized so that the field has constant variance one, that is:
\be \label{eq:normvar}
1=\Var\kop f(P)\pok =\sum_{l=|s|}^\infty 
\frac{c_l^2}{2}, 
\ee
for all $p\in M$. In addition, we will assume that the $c_l$ are such that the series in \cref{eq:f} converges in the $\mC^\infty$ sense\footnote{Actually, $\mC^1$ would be enough for most, if not all, of our purposes.}. For instance, if $c_l=0$ for all but a finite subset of indices $l\in \N$, then $f$ is automatically smooth, given that $D_{m,s}^l\in \mC^\infty(SO(3),\C)$ for all admissible choices of $l,m,s$.
A key parameter in our study is the number 
\be 
\xi^2:=\sum_{l=|s|}^\infty 
\frac{c_l^2}2\frac{(l(l+1)-s^2)}{2}>0
\ee
{that represents the \emph{frequency} of $f$, and }whose meaning will be explained below, see \cref{sec:nohomo}.
{The two parameters $\xi$ and $s$, the latter being called the \emph{spin} of $f$, fully characterize the 
metric induced by the field, see \cref{eq:metric}.}

\subsubsection{Spin-s fields.}
The Gaussian field defined in \cref{eq:f} is a left-invariant \emph{real spin-$s$ function} on $SO(3)$. 
Its characterizing properties are recalled in \cref{sec:propf}. It is the same model considered in \cite{elk2024PistolatoStecconi,articolo_dei_fisici_published}. Its complex counterpart $X$ has been more widely studied, see e.g.  \cite{stecconi2021isotropic,geospin2022LMRStec,BR13,libro,geller2008spin}, where it has been equivalently defined as a Gaussian isotropic \emph{spin-weighted function on $S^2$}, that is, a section $\sigma$ of the spin-s bundle $\spin{s}\to S^2$ (see the \emph{pull-back correspondence} in \cite{stecconi2021isotropic,geospin2022LMRStec,BR13}). It's important to remark that $f$ and $X$ are ---as processes--- fully correlated. Therefore, despite being defined on different spaces, we can say that $f$, $X$ and $\sigma$ carry the same statistical information. 

\subsection{The reference geometry}\label{subsec:refgeo}
In this paper, we consider the orthogonal group $SO(3)$ as a Riemannian manifold $(M,g)$, where the metric $g$ is the one induced by the embedding
\be 
\iota: SO(3)\hookrightarrow \R^6, \quad P\mapsto (Pe_2, Pe_3),
\ee
which sends an orthogonal matrix to the vector formed by its last two columns\footnote{Taking the fist two columns instead than the last would yield the same metric. This choice will come useful later, for notational reasons.}. We remark that $\iota$ is an isometry. 
Moreover, we define the projection onto the sphere 
\be 
\pi\colon SO(3)\to S^2, \quad P\mapsto Pe_3,
\ee
which is a Riemannian submersion, with fibers being circles of length $2\pi$, see \cite{stecconi2021isotropic}. {We refer to the monograph \cite{leeriemann}, for definitions.} {In this work, we study subsets $B\subset SO(3)$ that are union of fibers of $\pi$, i.e. $B = \pi^{-1}(\pi(B))$.} This is a natural choice, given that $f$ is associated to a spin-weighted function $\sigma$ on $S^2$.

\subsection{Main results}
It is immediate to see that a field $f$ constructed as in \cref{eq:f} and satisfying \cref{eq:normvar} always falls precisely in the setting of \cite{cgv2025StecconiTodino}, in which the boundary volume measure (area, in our case)
\be \label{eq:length}
\lft(B):=\mathrm{Area}(f^{-1}(t)\cap B), \quad \forall B\subset SO(3) \ \text{measurable,}
\ee
is studied through its Wiener-\^{I}to chaos decomposition, see \cref{subsec:prelim} for needed preliminaries. Here, the Area in $SO(3)$ stands for the $2$-dimensional Hausdorff measure associated to the metric $g$. The purpose of this paper is to specify the general formula given in \cite{cgv2025StecconiTodino} to the case of a real spin-s field $f$ as in \eqref{eq:f}. Indeed, one of the main novelties of such formula (see \cref{eq:nodalchaosAT}) is that it applies to non-isotropic or, more precisely, non-homothetic, Gaussian fields on a smooth manifold and $f$ is of such kind, see \cref{sec:invprop} below. 

Our main result is a chaos decomposition (of the push-forward via $\pi$) of the boundary area at level $t$ of spin-$s$ field as in \eqref{eq:f}:
\begin{equation}\label{eq:caosdec}
    \lft (\pi^{-1}(D) )= \E[\lft(\pi^{-1}(D))] + \sum_{q=1}^{\infty}  \lft[2q](\pi^{-1}(D)),
\end{equation}
where the series converges in $L^2(\P)$, and all terms are pairwise uncorrelated and belong to different Wiener-\^{I}to chaos spaces. More precisely, \cref{eq:caosdec} is understood as a decomposition of the random measure $\lft\pi^{-1}: D \mapsto \lft(\pi^{-1}(D))$ defined on Borel sets $D\subset S^2$. Moreover, the mapping $D\mapsto \lft[2q](\pi^{-1}(D))$ defines an (absolutely continuous) random measure on $S^2$ for each $q$, as shown in \cite{cgv2025StecconiTodino}. 

The expectation measure $\E[\lft]=\lft[0]$ was already studied in \cite[Theorem 1.1]{elk2024PistolatoStecconi} (see also \cref{eq:expectation} below); here, we extend such result to all $q\ge 1$. 
\cref{eq:expectation} (\cite{elk2024PistolatoStecconi}) shows, in particular, that the expectation is of order $\xi$. For this reason, we present our formulas rescaling them by $\xi$.
\begin{theorem}\label{thm:3}
If $s\neq 0$, then for all $q\in\N$ we have
\begin{multline}
     \frac{\m L_{f-t}(\pi^{-1}(D))}{\xi}[q]
=  \sum_{\substack{a,b\in \N,a+b=\!\frac q2
\\ }}  e^{-\frac{t^2}{2}}
\immondo_t\tyu a,b,  {\frac{s^2}{\xi^2}}\uyt  \\ \times
\int_{D}  
\Lag_{a}(\|f\|_{\spin{s}_x}^2)
\Lag_b\tyu \frac{\|\nabla^H_xf\|^2}{\xi^2}\uyt
\dd x,
\end{multline}
\end{theorem}
In the case $s=0$, the formula is slightly different, due to the fact that the the field $f$ is degenerate. Indeed, the area should actually be studied as a length. We discuss a comparison between the two formulas in \cref{sec:comments}.
\begin{theorem}\label{prop:s0}
    If $s=0$, then there exists an isotropic Gaussian field $\phi$ on $S^2$ such that $f=\phi\circ\pi$, and for all $q\in \N$ we have
    \begin{multline} 
\frac{\m L_{f-t}(\pi^{-1}(D))}{\xi}[q]=
 \sum_{a,b\in \N,a+b=\!\frac{q}2}
e^{-\frac{t^2}{2}}
\immondo_t(a,b,0)
\\
\times 
\int_D
 2\pi\cdot H_{2a}(\phi(x))
 \Sigma_b\tyu \frac{\|\nabla_x \phi\|^2}{\xi^2}\uyt 
 dx.
    \end{multline}
\end{theorem}
The coefficient $\immondo_t$ is defined in \cref{eq:immondo}: it's a polynomial in $t$ and an analytic function of $\frac{s}{\xi}$ (defined in terms of the hypergeometric function). $H_{2a}$ is the Hermite polynomial of degree $2a$, see \eqref{eq:H}. $\Lag_a$ are defined in \cref{def:sigma} by averaging $H_{2a}$ in all directions; in particular, they are univariate polynomials that form an orthogonal family for the $\chi_2^2$ distributions, see \cref{prop:ortogchi}.
The remaining terms to explain are the arguments of such polynomials: $\|f\|_{\spin{s}}$ and $\|\nabla^H f\|$, which are introduced in details, respectively, in \cref{def:spinsnorm} and \cref{def:verHor}. Here, we just mention that they have a natural interpretation in terms of the section $\sigma$ of the spin-s bundle $\spin{s}$, associated to $f$, and the derivative of $f$ along spherical directions.

\subsection{Comments on the results}\label{sec:comments}

\subsubsection{On the originality of \cref{prop:s0}.}
The formula in \cref{prop:s0} is essentially equivalent (up to the constant factor $2\pi$) to the one given for the nodal length ($t=0$) of Gaussian spherical harmonics in the $2$-dimensional case of \cite[Theorem 3.4]{marinucci2023laguerre} (it corresponds to consider $s=0$ and $c_l=\delta_l^\ell c_\ell$ for some $\ell \in \N$ in  \cref{eq:f}). Indeed, one can recognize that $\Sigma_a$ is a constant multiple of the generalized Laguerre polynomials appearing in \cite[Theorem 3.4]{marinucci2023laguerre}, see \cite[Section 1.2.8]{cgv2025StecconiTodino} for a detailed comparison. That considered, the original content of \cref{prop:s0} resides mostly in the identification of the coefficient $\immondo_t(a,b,0)$ as the evaluation at $0$ of $\immondo_t(a,b,s^2/\xi^2)$ in  \cref{thm:3}, valid for general spin $s$.
\subsubsection{Non-homothetic setting.} \label{sec:nohomo}
The complexity of the above formula (the coefficient $\immondo$ involves hypergeometric functions) is due to the fact that the field $f$ might not be homothethic, that is, $s\neq \xi$. In fact, in our setting the reasonable intuition is that $\frac{s}{\xi}\to 0$ (see \cite{elk2024PistolatoStecconi}). The homothetic case $s=\xi$ can still be realized by some very special spin-$s$ fields and in such case, the formula simplifies according to 
\be 
\immondo_t(a,b,1)=\sum_{i=0}^a \frac{B\left(i + \frac{1}{2},\, \beta + 1\right)}{_2F_1(\a-i,i;\a+\frac12;1)} \binom{2(i+\beta)}{2i}\cdot \frac{\coeff(\a-i,\beta+i)}{s_3}\frac{H_{2\a-2i}(t)}{H_{2a-2i}(0)},
\ee
since $\prescript{}{2}{F}_1(a,b;c;0)=1$. The fact that $s\neq \xi$ implies that the Adler-Taylor metric of $f$ is not a multiple of the reference metric on $SO(3)$; this \emph{anisotropy}, prevents from applying more classical formulas. The results of this paper rely instead on the recent formula proved in \cite{cgv2025StecconiTodino}.

\subsubsection{Can one read $s$ from the high frequency limit?} \label{rem:asymptotics}

More relevant to our analysis is the behavior as $\frac{s}{\xi} \to 0$. This is particularly complicated to understand, since $\prescript{}{2}{F}_1(a,b;c;z)$ is singular at $|z|\to 1^-$. 

\begin{remark} 
A deep investigation is beyond the scope of this paper. Nevertheless, we believe it to be of great interest as it will entail understanding the asymptotics of these formulas as $\xi \to \infty$ (high frequency limit). 
\end{remark}

Nevertheless, we can already make some preliminary yet interesting observations.
The structure of the formula in \cref{thm:3}, for the \emph{Area} of the set $$\{P\in SO(3)\colon f(P)=t, \pi(P)\in D\}$$ is similar to that of the \emph{Length} of $\kop x\in D: \phi(x)=t\pok$ for a Gaussian field $\phi$ on $S^2$, that indeed corresponds to the case $s=0$, see \cref{prop:s0}.
\begin{remark}\label{rem:mark}
$\Sigma_a(\|f\|^2_{\spin{s}_x})$ in \cref{thm:3}, is replaced by $2\pi\cdot H_{2a}(\phi(x)) $  in \cref{prop:s0}.
\end{remark}
Indeed, on one hand the former is obtained by averaging $H_{2a}(f(P))$ over  $P\in \pi^{-1}(x)$ (see \cref{def:sigma}). On the other side, if $s=0$, the field $f$ turns out to be constant on the fibers of $\pi\colon SO(3)\to S^2$; thus, given that each fiber has length $2\pi$ and that $\pi$ is a Riemannian submersion, we can express the \emph{Area} of a level of $f$ as the \emph{Length} of the level of $\phi$ via the identity
\be 
\m L_{f-t}(\pi^{-1}(D))=2\pi\cdot \m L_{\phi-t}(D), \quad \text{ if $s=0$}.
\ee

The similarity between the decomposition of the area when $s\neq 0$, and that of the length, if $s=0$, could be intuitively explained by the fact that in the limit $\xi\to +\infty$ the field behaves as if $s=0$, given that the coefficient $\immondo_t$ depends on $\frac{s}{\xi}$. This intuition is again confirmed by the scaling limit proved in \cite{geospin2022LMRStec}, from which the authors deduce that in the regime ``$s$ fixed, $\xi\to +\infty$'' the leading order in the asymptotic behavior of the expectation of a large class of geometric functionals is not sensitive to the spin.
The same phenomenon emerges once more from the formulas for the expectation of $\lft(B)=2\m L_2(\{f\ge t\}\cap B)$, computed in \cite{elk2024PistolatoStecconi}. \begin{theorem}[\!\!{\cite[Theorem 1.1]{elk2024PistolatoStecconi}}]\label{thm:elk}
    \begin{equation} \label{eq:expectation}
    \E\qwe\frac{\lft\tyu \pi^{-1}(D)\uyt}{\xi}\ewq = 2
    \vol{}(D) \ e^{-t^2/2} \  \begin{cases}
          \frac{\arcsin \sqrt{1-\frac{s^2}{\xi^2}}}{\sqrt{1-\frac{s^2}{\xi^2}}} + \frac{|s|}{\xi} \hspace{1cm} & \text{if } \xi^2 > s^2,  \\
         2 & \text{if } \xi^2 = s^2,\\
         \frac{\arcsinh \sqrt{\frac{s^2}{\xi^2}-1}}{\sqrt{\frac{s^2}{\xi^2}-1}} + \frac{|s|}{\xi}  & \text{if } \xi^2 < s^2.
    \end{cases}
\end{equation}
\end{theorem}
The same result is recovered by \cref{thm:3} and \cref{prop:s0} with $q=0$, as the two terms mentioned in \cref{rem:mark} coincide for $a=0$.

Once again, the leading order in the limit $\xi\to +\infty$ of \cref{eq:expectation} is the same for all $s$. \cref{thm:3} and \cref{prop:s0} extend such phenomenon to all chaos coefficients
\be 
\immondo_t\tyu a,b,\frac{s}{\xi}\uyt \xrightarrow[\xi\to+\infty]{}\immondo_t\tyu a,b,0\uyt,
\ee
implying that such information is not sensitive to the spin, at the leading order. However, in \cref{thm:3} something new can be observed: the term $\Sigma_a(\|f\|^2_{\spin{s}_x})$, which can be deduced by the chaos components of order $q\ge 2$ in the asymptotic $\xi\to +\infty$, contains some information about the spin parameter $s$. 
For instance, by \cref{prop:ortogchi}, for all integers $s \neq 0$ we have that\footnote{The innermost inequality can be deduced from well known properties of the Gamma functions; with equality if $a=0$.}
\be 
\Var\kop \Sigma_{a}\tyu \|f\|_{\spin{s}}^2\uyt \pok= (2a)! 4\pi^{\frac 32}\frac{\Gamma\tyu a+\frac12\uyt}{a!} \le 
(2a)!4\pi^2 
=\Var\kop 2\pi \cdot H_{2a}(\phi(x))\pok.
\ee
\begin{remark} 
Such relation implies that the variance of higher order chaoses is sensitive to the spin: the asymptotic behavior as $\xi\to+\infty$  when $s\neq 0$ (e.g. $s=2$ for CMB applications), is different from that exhibited when $s=0$. The intuition maturated in past works \cite{geospin2022LMRStec,articolo_dei_fisici_published,elk2024PistolatoStecconi} is that it should be \emph{safe} to replace $s=0$ in an analysis of the high frequency limit regime. In conclusion of this work, we can say that this intuition is wrong! Despite such principle applies for the expectation, the spin, e.g. $s=0$ or $s=2$, changes the leading order of the variance.
\end{remark}

\subsubsection{Variance.}
In \cite[Lemma 4.1]{cgv2025StecconiTodino}, relying on the representation of $\Sigma_a$ in terms of Hermite polynomials, the authors include a formula to compute covariances of random variables of the form $\Sigma_a(\chi^2_2)$, by specializing diagram formulas, see  \cite[Lemma 5.2]{CARAMELLINO2024110239} and \cite[Section 4.3.1]{libro}.
This implies that the chaos decomposition established in \cref{thm:3} and \cref{prop:s0} can be exploited to efficiently compute the variance of the level volume.

\section{Further properties of spin fields} \label{sec:propf}
We will denote the covariance function of $f$ by
\be 
C(P,Q):=\E\kop f(P)f(Q)\pok.
\ee
Clearly, the normalization at \cref{eq:normvar} above means that $C(P,P)=1$, for all $P\in SO(3)$.
Finally, the Adler-Taylor metric of $f$ is the Riemannian metric $g^f$ defined as 
\be 
g^f_P=\E\kop d_Pf \otimes d_Pf\pok \colon T_PM\times T_PM \to \R.
\ee 
The comparison between $g$ and $g^f$ has a crucial role in \cite{cgv2025StecconiTodino}, as well as in \cite{elk2024PistolatoStecconi}. This latter paper already contains a detailed study of $g^f$ for a spin-s field $f$ constructed as in \cref{eq:f}, and we will rely heavily on that in what follows.

\subsection{Invariance properties}\label{sec:invprop}
The right action of $S^1=O(2)$ on $SO(3)$ has a special role in the study of spin fields, that is, for any $\psi \in \R$, we denote by 
\be 
R_3(\psi):=\begin{pmatrix}
    \cos \psi & -\sin \psi & 0 \\
    \sin \psi & \cos \psi & 0 \\
    0 & 0 & 1
\end{pmatrix}
\ee
the rotation of angle $\psi$ around the $e_3$ axis. Then, any realization of the field $X$ satisfies the (deterministic) identity: \begin{equation}
    X\tyu P R_3(\psi)\uyt=X(P)e^{-is\psi},
\end{equation}
which is the defining property of \emph{spin-$s$ functions} (i.e., functions on $SO(3)$ with (right) spin equal to $-s$, in the language of \cite{stecconi2021isotropic}). 
Consequently the field $f=\Re(X)$ inherits the following property:
\begin{proposition}
\be\label{eq:spinproperty} 
f\tyu P R_3(\psi) \uyt = \begin{cases} f(P)\cos\tyu s \psi\uyt + f\tyu P R_3\tyu\frac{\pi}{2 s}\uyt\uyt \sin\tyu s\psi \uyt, \quad &\text{if $s\neq 0$};
\\
f(P), \quad &\text{if $s= 0$}.
\end{cases}
\ee
and if $s\neq 0$, then: $f(P)=-f\tyu P R_3\tyu\frac{\pi}{ s}\uyt\uyt$ and $f\tyu P R_3\tyu\frac{\pi}{2 s}\uyt\uyt$ are independent.
\end{proposition}
In particular, it follows that $f$ is invariant in law under such action:
\be 
f\tyu (\cdot) R_3(\psi) \uyt \sim f(\cdot).
\ee
Moreover, the fields $f$ and $X$ are left-invariant in law: for all $R\in SO(3)$
\be 
f\tyu R (\cdot) \uyt \sim f(\cdot).
\ee
Both invariances have to be intended as the analogous invariance of the covariance function, meaning that 
\be \label{eq:covinva}
C(P,Q)=C\tyu \1, R_3(-\psi)P^{-1}Q R_3(\psi)\uyt ,
\ee
for all pairs $P,Q\in SO(3)$ and $\psi\in \R$. See \cite{geospin2022LMRStec,stecconi2021isotropic}. 
\begin{remark}
We will avoid using the word \emph{isotropic} since it leads to dangerous ambiguities, as explained in \cite{stecconi2021isotropic}. We stress the fact that $f$ is not invariant under all isometries of $SO(3)$, as it is not right-invariant. For the same reason it is not possible to express its covariance function as a a function of the distance: an expression like $C(P,Q)=C(\dist PQ)$ is never possible (not even for $s=0$, a degenerate case in view of \cref{eq:spinproperty}). Instead, one has to be content with \cref{eq:covinva}, which can be further reduced in terms of Euler angles as in \cite{geospin2022LMRStec,stecconi2021isotropic} and reported in \cite[Eq. (3.6)]{elk2024PistolatoStecconi}.
\end{remark}

\subsection{The norm of the spin section} 

Recall that $f(P)=\Re (X(P))$, and that the invariance property (\cref{prop:inva}) comes from the identity \cref{eq:spinproperty}
\be 
X(PR_3(\psi))=X(P) e^{-is\psi},     
\ee
valid for all $\psi\in \R$, $P\in M$. The previous identity implies that the (random) norm {\begin{equation}\label{eq:randomnorm}
    |X(P)|^2 = |\Re (X(P))|^2 + |\Im (X(P))|^2 = |f(P)|^2 + \left|f\left(PR_3\left(\frac{\pi}{2s}\right)\right)\right|^2
\end{equation}} 
only depends on $x=\pi(P)\in S^2$. Therefore, the following definition is well-posed.
\begin{definition}\label{def:spinsnorm}
Given $x\in S^2$, we define the spin-$s$ norm of $f$ at the point $x\in S^2$ as \begin{equation}
    \|f\|_{\spin{s}_x}:=|X(P)|
\end{equation} 
for any $P\in SO(3)$ such that $\pi(P) = P\cdot e_3=x$.
\end{definition}
The meaning of $\|f\|_{\spin{s}_x}$ is that if we identify $f=\Re(X)$ with the corresponding section $s_f$ of the spin-s line bundle $\spin{s}\to S^2$ (as in \cite{stecconi2021isotropic,geospin2022LMRStec}), then 
\be 
\|f\|_{\spin{s}_x}=|X(P)|=\|s_f(x)\|_{\spin{s}_x},
\ee
where $P\cdot e_3=x$.

\section{A deeper investigation of the geometry induced by a spin field} \label{sec:propgf}
\subsection{Remarks on the choice of reference geometry}
Recall \cref{subsec:refgeo}.
\begin{remark}
Notice that the image of the isometry $\iota\colon SO(3)\hookrightarrow \R^6$ is the total space of the unit tangent bundle of the sphere (see \cite{elk2024PistolatoStecconi}), that we denote by 
\be 
\m{T}=\kop (v,x)\in \R^3\times \R^3 : x\in S^2,\ v \in T_xS^2,\ \|v\|=1\pok
\ee 
as in \cite{stecconi2021isotropic} and \cite{geospin2022LMRStec}. 
Under such identification, the right action of $R_3(\psi)$ acts precisely by rotating the fibers in anti-clockwise fashion: {if $(v,x)=\iota(P)$, then}
\be \label{eq:fiberCircle}
{\iota \colon} PR_3(\psi)\mapsto (e^{-i \psi}v,x), \qquad {\forall \psi\in\R},
\ee
{where $e^{-i \psi}v$ denotes the canonical complex multiplication in the fibers of $\m T$, see \cite{geospin2022LMRStec}.}
\end{remark}
\begin{remark} 
The latter observation justifies the choice of such geometry in view of the alternative descriptions of spin fields as sections of $\spin{s}$, for which we refer to \cite{geospin2022LMRStec,stecconi2021isotropic}. Such point of view was introduced in \cite{geller2008spin} which consider spin fields functions as random sections of a complex line bundle on $S^2$. The relation between random sections of line bundles on $S^2$ and random functions on $SO(3)$ has been described in detail in the related literature, first in \cite{BR13}, and later in 
\cite{geospin2022LMRStec,stecconi2021isotropic}, where the line bundle of spin $s$ is identified as the $s^{th}$ complex tensor power of $\m{T}$, from which the notation $\spin{s}$. There is nothing deep about it, given that a well known topological fact is that all complex line bundles on $S^2$ are characterized up to isomorphisms by their Chern class, or equivalently their Euler characteristic, which is an integer $2s\in \Z$. As a consequence, the complete list of isomorphism classes of complex line bundles on $S^2$ is $\kop \spin{s}:s\in \frac12\Z \pok$, including also half-integer spin, but nothing more.
\end{remark}
\subsection{The geometry of the field}
\subsubsection{Euler coordinates.}
We use the parametrization of $SO(3)$ given by Euler angles, following the notations and conventions of \cite{libro,geospin2022LMRStec,stecconi2021isotropic,elk2024PistolatoStecconi}.
\begin{definition}\label{def:eulerangles}
For any $\f, \theta, \psi\in \R$, we denote $R(\f,\theta,\psi):=R_3(\f)R_2(\theta)R_3(\psi)$, where
\bega 
R_3(\f):=\begin{pmatrix}
\cos \f & -\sin \f & 0 \\
\sin \f & \cos \f & 0 \\
0 & 0 & 1
\end{pmatrix}, \quad R_2(\theta):=\begin{pmatrix}
\cos \theta & 0 & \sin \theta \\
0 & 1 & 0 \\
-\sin \theta & 0 & \cos \theta \\
\end{pmatrix}
\eega
In particular, $R(0,\frac \pi 2, 0)=:P_0$.
\end{definition}
Such function defines a diffeomorphism \be 
R^{-1}\colon U:=SO(3)\smallsetminus \kop R_2(\theta): \theta\in \R\pok\to (-\pi,\pi)\times (0,\pi)\times (-\pi,\pi)
\ee
which will be our preferred chart for $SO(3)$. Notice that its domain $U$ is a dense open subset of full measure, and that $P_0 \in U$. Note also that $\1\notin U$, which is the reason why we don't use such point as $P_0$.
\subsubsection{The Gram matrix.}
The Gram matrix of the Adler-Taylor metric
of $f$
at a point $P\in M$, in the Euler angles coordinates $P=R(\f,\theta,\psi)$, see \cref{def:eulerangles}, is given by the following formula proved in \cite{elk2024PistolatoStecconi}:
\begin{equation}\label{eq:metric}
 \Sigma_{(\xi,s)}(\theta) 
      =  \begin{pmatrix}
         \xi^2 \sin^2(\h)+s^2 \cos^2(\h) & \zero & s^2 \cos(\h) \\
         \zero & \xi^2 & \zero \\
         s^2 \cos(\h) & \zero & s^2
     \end{pmatrix}, \; \Sigma_{(\xi,s)}(\pi/2) 
      =  \begin{pmatrix}
         \xi^2  & \zero & \zero \\
         \zero & \xi^2 & \zero \\
         \zero & \zero & s^2
     \end{pmatrix}.
\end{equation}
 Moreover, the standard metric $g$ on $SO(3)$ has Gram matrix $\Sigma_{1,1}(\h)$. 
\subsubsection{Left-invariance of both geometries.}
 An important thing to remember is that both metrics $g$ and $g^f$ are left-invariant and invariant under right multiplication by $R_3(\psi)$, that is, for any $R\in SO(3)$ and any $\psi\in \R$, the map 
 \be 
\Phi: M\to M, \quad \Phi(P)=RPR_3(\psi),
 \ee
 is an isometry of both $(M,g)$ to itself and $(M,g^f)$ to itself. This means precisely that the differential \be 
 d_P\Phi: (T_PM,g_P^{i})\to (T_{\Phi(P)}M,g_{\Phi(P)}^{i})
 \ee
 is a linear isometry for both $g^i\in \{g,g^f\}$.

\subsubsection{Decomposition of the gradient.}

\begin{definition}\label{def:verHor}
    Let $V$ be the unit vector field on $M$ defined as
    \be 
V(P):=\frac{\dd }{\dd \psi}\Big|_{\psi=0} P R_3(\psi) \in T_PM.
    \ee
    For any $f\in \mC^\infty$, we will denote
    \be 
\de_\psi f(P):=\langle d_Pf, V(P)\rangle.
    \ee
    Let $\Ho(P):=V(P)^{\perp_g}\subset T_PM$ be the orthogonal complement. 
    We call multiples of $V$ \emph{vertical} and elements of $H$ \emph{horizontal} and use the notation $u_V$ and $u_H$ for the orthogonal projections of a generic tangent vector $u\in T_PM$. Moreover, we denote $\nabla^Hf:=(\nabla f)_H$ and call it the \emph{horizontal gradient of $f$}.
\end{definition}
\begin{remark}
    The image of the curve $\{\iota(PR_3(\psi)):\psi\in \R\}$ is a circle of radius $1$ in $\R^6$, see \eqref{eq:fiberCircle}, implying that $V(P)$ has length one.
\end{remark}

It follows that the gradient (with respect to the original metric $g$) of $f$ can be written as 
\be \label{eq:decNabla}
\nabla f=\nabla^H f+\de_\psi f \cdot V,
\ee
where the two terms in the sum are orthogonal with respect to both $g$ and $g^f$. 
\begin{remark}[Scope of the decomposition - how to read this section if you are scared of geometry]
{In the next section, we will need to integrate functions of the norm of $\nabla f$ (with respect to the metric induced by the field, $g^f$, see \cref{eq:nodalchaosAT}). The results that follow aim to express it in terms of the norm of $\nabla^H f$ (with respect to the original metric  $g$) and $\de_\psi f$. Decomposition \eqref{eq:decNabla} states that the random (Gaussian) gradient $\nabla_P f \randin T_PM$ admits unique projections onto $  \Ho(P)$ and $ V(P)\mathbb R$, and that the two spaces are orthogonal with respect to both $g$ and $g^f$. This implies that $ \nabla^H_P f \randin \Ho(P)$ and  $\de_\psi fV(P) \randin V(P)\mathbb R$ are independent Gaussian. A careful investigation of their norms, carried in the following, allows us to show that they are, as Gaussian vectors, standard with respect to the metric $g^f$, and, up to a constant, standard with respect to the original metric, see \cref{def:standGaussvec} and \cref{prop:normaljet}. Then, \cref{prop:samenorm} will allow us to carry out the integration fiber-wise, exploiting the results on the spherical chaos proved in the previous section, since the contribution to $\|\nabla_P f\|_{g^f}$ of $\|\nabla_P^H f\|$ is constant along the fiber. At the end of the section, we will remark that the contribution of $\de_\psi f$ fully depends on the field $f$, and its so-called \emph{spin-$s$ norm}, see \cref{def:spinsnorm}.}
\end{remark}

{The following result characterizes the derivative of the field along the fiber, $\de_\psi f$, in terms of the field itself.}
\begin{proposition}\label{prop:inva}
\be 
\frac{1}{s}\de_\psi f
=
f\tyu P R_3\tyu\frac{\pi}{2 s}\uyt\uyt.
\ee
\end{proposition}
\begin{proof} An immediate application of the spin-$s$ property yields \begin{equation}
    \frac{1}{s}\de_\psi f
=
\frac{1}{s}\frac{\de }{\de \psi}\Big|_{\psi=0}f(PR_3(\psi))=\langle d_Pf,\frac{1}{s}V(P)\rangle=f\tyu P R_3\tyu\frac{\pi}{2 s}\uyt\uyt.
\end{equation}
\end{proof}
The following proposition provides an intrinsic way to express the relation between the Gram matrices of $g$ and of $g^f$.
\begin{proposition}\label{prop:ortogonality}
$\Ho(P)$ and $V(P)$ are orthogonal with respect to $g^f$. Moreover, for any $u\in T_PM$ we have
\be 
\|u\|^2_{g^f}=\xi^2\|u_H\|^2+s^2\|u_V\|^2.
\ee
\end{proposition}
\begin{proof}
    It is sufficient to check it for $P_0$, see \cref{eq:metric}, and in such case it is obvious since the matrices $\Sigma$ are simultaneously diagonalizable. 
\end{proof}
\subsubsection{What do we like about this decomposition?}
The map $SO(3)\to S^2$ given by $P\mapsto P\cdot e_3$ corresponds, under $\iota$, to the canonical projection $\pi:\m T\to S^2$ of the unit tangent bundle. In particular, it defines a Riemannian circle bundle (it is a Riemannian submersion) over $S^2$ that is intrinsically related to spin-s functions (see \cite{stecconi2021isotropic}). Taking the differential we obtain 
a very clear picture explaining the meaning of the decomposition into horizontal and vertical. Let $\pi(P)=x$, then
\be 
D_P\pi\colon T_PM\to T_xS^2, \quad \ker(D_P\pi)=\R\cdot V(P), \quad \ker(D_P\pi)^\perp=\Ho(P),
\ee
implying that $D_P\pi$ restricts to an isometry between $\Ho(P)$ and $T_xS^2$. 
\subsubsection{The first jet.}\label{subsec:firstJet}
We can decompose the first jet $j^1_Pf=(f(P),\nabla_Pf)$ of $f$ as follows:
\bega 
j^1_Pf&=
\begin{pmatrix}
    f(P), & f\tyu PR_3\tyu \frac{\pi}{2s}\uyt\uyt \cdot sV(P) + \frac1{\xi}\nabla^H_Pf \cdot \xi
\end{pmatrix}
\\
&= \Lambda_P
\begin{pmatrix}
    f(P), & f\tyu PR_3\tyu \frac{\pi}{2s}\uyt\uyt, & \frac1{\xi}\nabla^H_Pf 
\end{pmatrix},
\eega
where $\Lambda_P\colon \R\times \R\times \Ho(P)\to \R\times T_PM$ is the linear change of basis such that
\bega 
(1,0,0)&\mapsto (1,0) 
\\
(0,1,0) &\mapsto (0, sV(P))
&
\\
(0,0,u_H)&\mapsto (0,u_H\cdot \xi).
\eega
\begin{remark}
  The operator  $\Lambda_P$ is essentially equivalent to the \emph{pointwise frequency endomorphism} defined in \cite{cgv2025StecconiTodino}.
\end{remark}
\begin{definition}
    We will write 
    \be 
\Lambda^*j_P^1 f:=\begin{pmatrix}
    f(P), & f\tyu PR_3\tyu \frac{\pi}{2s}\uyt\uyt, & \frac1{\xi}\nabla^H_Pf 
\end{pmatrix}.
    \ee
\end{definition}
What is good about such representation of the jet, and thus of the gradient is the following.
\begin{definition}\label{def:standGaussvec}
    Let $V$ be a vector space of rank $n$, endowed with a positive-definite scalar product $g=\langle\cdot \, ,\cdot\rangle$. Then, a random element $X\randin V$ is said to be a \emph{standard Gaussian vector with respect to} $g$ if there exists an orthonormal basis $\m B$ of $V$ such that ${[X]_{\m B}} \randin \R^n$ (the vector of the components of $V$ with respect to the basis $\m B$) is a (real) standard Gaussian vector. 
\end{definition}
\begin{remark}
    We remark that it is equivalent to ask that ${[X]_{\m B}} \randin \R^n$ is a (real) standard Gaussian vector for \emph{any} orthonormal basis of $V$.
\end{remark}
\begin{proposition}\label{prop:normaljet}
The random vector 
\be 
\Lambda^*j_P^1 f \ \randin \ \R\times \R \times \Ho(P)
\ee
is a standard Gaussian vector (when $\Ho(P)$ has the metric $g_P$).
\end{proposition}
\subsubsection{Spin-s invariance of the jet.}
Because of the spin-s property of $f$,  see \cref{eq:spinproperty}, we deduce the following almost sure invariance of the gradient norm.
\begin{proposition}\label{prop:samenorm}
The following identity holds \emph{surely}:
\be 
\|\nabla^H_{PR_3(\psi)}f \|=\|\nabla^H_{P}f \|,
\ee
for all $\psi\in \R$. In other words, $\|\nabla^H_{P}f \|$ descends to a function on $S^2$, which we will denote as $\|\nabla^H_{x}f \|:=\|\nabla^H_{P}f \|$ whenever $x = \pi(P)$.
\end{proposition}
\begin{proof}
The multiplication by $R_3(\psi)$ (that is the flow at time $\psi$ of the vector field $V$) is an isometry of $g$ (and also of $g^f$). By definition, its differential $D_P\colon T_PM\to T_{PR_3(\psi)}M$ sends $V(P)$ to $V(PR_3(\psi))$. Therefore, it also sends $\Ho(P)$ to $\Ho{(PR_3(\psi))}$, isometrically. Note that then $D_P$ is a linear isometry, meaning that $D_P^*=D_P^{-1}$. The chain rule (dualized) gives
\be 
D_P^{-1}\tyu \nabla_{PR_3(\psi)} f\uyt= \nabla_Pf,
\ee
therefore the two vectors have the same norm.
\end{proof}

\subsection{Bonus material: $\Ho$ is a trivial bundle}
The \emph{horizontal} bundle $\Ho$ defined in \cref{def:verHor} is the lift to $SO(3)$ of the tangent bundle of $TS^2$. Despite the latter is non-trivial---by the famous Hairy Ball Theorem---it turns out that $\Ho$ is a trivial bundle on $SO(3)$. This is not a coincidence, in fact it is practically a tautology, given that $SO(3)$ is the frame bundle of $S^2$. We won't rely heavily on this fact, but we think it is worth noticing. So, here's a proof.
\begin{proposition}
The bundle $\Ho\to SO(3)$ has two orthonormal smooth sections $\hat e_i\colon SO(3)\to H$, for $i=1,2$.
\be
\hat e_i(P)=(D_P\pi)^{-1}(P\cdot e_i), 
\ee
where $\pi\colon SO(3)\to S^2$ is the map $\pi(P)=P\cdot e_3$.
\end{proposition}
\begin{proof}
Note that $P\cdot e_i$ form an orthonormal basis of $T_{P\cdot e_3}S^2$. Moreover, we already noted that $D_P\pi\colon \Ho(P)\to T_{P\cdot e_3}S^2$ is an isometry (i.e., $\pi$ is a Riemannian submersion).
\end{proof}
The take-home message of this subsection is the following. In view of the existence of the smooth global sections $\hat{e}_1$, $\hat{e}_2$, we can identify $\Ho(P)\cong \R^2$ isometrically, and interpret \cref{prop:normaljet} as the statement that $\Lambda_P^*j^1_Pf\sim \m N(0,\1_{4})$. 

\section{Proof of the chaos decomposition} \label{sec:chaos} 

\subsection{What is chaos? Preliminaries} \label{subsec:prelim}
The field defined in \eqref{eq:f} is a linear combination of elements of the sequence $\{\Re\gamma_{m,s}^l,\Im\gamma_{m,s}^l\}_{s\in\Z, l\in\N, -l\le m\le l}$, which is defined on some probability space $(\Omega,\mathcal F,\P)$ ($\mathcal F$ being generated by the Gaussian sequence), and verifies \begin{align}
    & \E[\Re\gamma_{m,s}^l] = \E[\Im\gamma_{m,s}^l] = 0 ,\quad \E[\Re\gamma_{m,s}^l\Im\gamma_{m,s}^l] = 0, \\
    & \E[\mathrm P \gamma_{m,s}^l \mathrm Q \gamma_{m',s}^{l'}] = \frac12 \, {\delta_l^{l'} \delta_m^{m'}  \delta_{\mathrm P}^{\mathrm Q}}.
\end{align}
for all $\mathrm P,\mathrm Q\in \{\Re,\Im\}$. It is a standard fact that the closure in $L^2(\P)$ of the space of real linear combinations of $\{\Re\gamma_{m,s}^l,\Im\gamma_{m,s}^l\}_{s,l,m}$, endowed with the scalar product $\langle\cdot,\cdot\rangle = \Cov(\cdot,\cdot)$, is a Gaussian Hilbert space, see e.g. \cite{janson,nourdinpeccatibook} for a full account; we denote it as $\mathcal H$. Let us define $H_q$, the $q^{th}$ Hermite polynomial, as 
\be\label{eq:H}
\sum_{q\in \N}H_q(x)\frac{t^q}{q!}=e^{tx-\frac{t^2}{2}}.
\ee
We refer to \cite{nourdinpeccatibook} for equivalent definitions. In particular, $H_q$ is a polynomial of degree $q$, and together they form a family of orthonormal polynomials for the Gaussian measure on $\R$. We define $\mathcal C_q$, the $q^{th}$ chaos, as the closed linear subspace of $L^2(\P)$ generated by the random variables  $\{H_q(\zeta) : \zeta\in\mathcal H, \|\zeta\|_{\mathcal H}=1\}$: \begin{equation}
    \mathcal C_q := \overline{\textrm{span}\{H_q(\zeta) : \zeta\in\mathcal H, \|\zeta\|_{\mathcal H}=1\}}^{L^2(\P)}.
\end{equation}
The orthogonality of the Hermite polynomials implies that $\mathcal C_q \perp \mathcal C_p$ for all $q\neq p$; moreover, $\mathcal H = \bigoplus_{q\in\N_0} \mathcal C_q$. This fact implies that any $\mathcal F$-measurable and square-integrable functional $F$ can be written in $L^2(\P)$ as \begin{equation}\label{eq:WIC}
    F = \sum_{q=0}^{\infty} F[q],
\end{equation}
where $ F[q] = \mathrm{proj}( F |\mathcal C_q)$ is the (unique) projection of $F$ onto $\mathcal C_q$. \cref{eq:WIC} is usually referred as the \emph{Wiener-It\^o chaos decomposition}, or briefly \emph{chaos decomposition}, of $F$. In particular, $\E[F] = F[0]$ and \begin{equation}
    \Var(F) = \sum_{q=1}^{\infty} \Var(F[q]) < \infty.
\end{equation} 
We remark that any $d$-dimensional subspace $\mathcal H'\subset \mathcal H$, i.e. generated by a finite orthonormal system $\{\eta_1,\ldots,\eta_d\}\subset \mathcal H$, is a proper Gaussian Hilbert space.\footnote{To exemplify the situation, in the following we consider $\eta = (f(P),\nabla f(P))\sim \mathcal N(0,C)$ for some $4\times4$ covariance matrix $C$, and functionals of the field and its gradient at a fixed point $P\in SO(3)$, i.e. $F=F(f(P),\nabla f(P))$. In view of \cref{eq:ROTAZIONEEEE}, it will be particularly easy to derive the chaos decomposition of functionals that only depend on a specific norm of $\nabla f(P)$.} Therefore, it admits its own chaotic decomposition $\mathcal H' = \bigoplus_{q\in\N_0} \mathcal C_q'$ into a direct sum of finite-dimensional (hence closed) spaces \begin{equation}
    \mathcal C_q' := \textrm{span}\{ H_q(\langle \eta, v\rangle) : v\in S^{d-1} \} ,
\end{equation} 
where $\eta = (\eta_1,\ldots,\eta_d) \sim\mathcal N(0,\1_d)$. In particular, $\m C_q'=\m C_q\cap \mathcal H'$. As a consequence, if $F=F(\eta)\in\mathcal H'$ only depends on the norm of $\eta$, we can write \begin{equation}\label{eq:ROTAZIONEEEE}
    F(\eta)[q] = c_q \int_{S^{d-1}} H_q(\langle \eta, v\rangle) dv, \quad \text{where } c_q := \frac{\E[F(\eta)H_q(\eta_1)]}{q!\beta(d,q)\vol{}(S^{d-1})}.
\end{equation}

\subsection{Spherical chaos}\label{sec:spherchaos}
As hinted by the previous discussion, and become more evident in the proofs, see e.g. \eqref{eq:nodalchaosAT}, a special role is played by polynomials introduced in the next definition. We reveal (and will prove in the next proposition) that they are orthogonal polynomials for the  $\chi^2$ distributions of parameter $n$.

\begin{definition}\label{def:sigma}
For any $n,b\in\N$, we define the polynomials $\Sigma_{n,b}\in \R[X]$ in one variable, of degree $b$ as
\be \label{eq:lag}
\Sigma_{n,b}(\|x\|^2):=S_{n,2b}(x):=\int_{S^{n-1}}H_{2b}\tyu \langle x,u\rangle\uyt \dd u,
\ee
where $x\in \R^n$. In what follows, we will denote $\Sigma_b:=\Sigma_{2,b}$, omitting the subscript when $n=2$. 
\end{definition}

The well-posedness of the above definition is due to the fact that $H_{2b}(t)$ is a polynomial in $t^2$, and that $S_{n,b}(x)$ a function of the norm of $x$ only, by rotational invariance. We remark the polynomials $S_{n,2b}$ have been already defined and analyzed in \cite{cgv2025StecconiTodino}, whose noteworthy properties are summarized in the next proposition.
\begin{proposition}\label{prop:ortogchi}
If $\eta\sim \m N(0,\1_n)$, then $\Sigma_{n,b}\tyu \|\eta\|^2\uyt$ is an element of the Wiener chaos of order $2b$, with variance (if $n=2)$
\be 
\sigma_{b}^2:=\Var\kop \Sigma_{b}\tyu \|\eta\|^2\uyt \pok= (2b)!\cdot 4\pi^{\frac 32}\frac{\Gamma\tyu b+\frac12\uyt}{b!} .
\ee Given $\F(\eta)\in L^2$, such that $\F\colon \R^n\to \R$ depends only on $\|\eta\|$, then the series
\be 
\F(\eta)= \sum_{b\in\N} a_b(\F)\Sigma_{n,b}\tyu \|\eta\|^2\uyt.
\ee
converges in $L^2$,
for some coefficients $a_b(\F)\in\R$ and all terms are pairwise uncorrelated.
\end{proposition} 
\begin{proof}
All follows from \cite[Theorem 3.4]{cgv2025StecconiTodino}. The only thing to prove is the variance:
\bega
\Var\kop \Sigma_{n,b}\tyu \|\eta\|^2\uyt \pok 
&=
\int_{S^{n-1}}\int_{S^{n-1}} (2b)! (\langle u,v\rangle)^{2b} \dd u \dd v
\\
&= (2b)!s_{n-1} 2\pi^{\frac{n-1}{2}}\frac{\Gamma\tyu \frac{2b+1}{2}\uyt}{\Gamma\tyu \frac{2b+n}{2}\uyt}
\\
&= (2b)!\frac{1}{\Gamma\tyu\frac n2\uyt} 4\pi^{n-\frac12}\frac{\Gamma\tyu \frac{2b+1}{2}\uyt}{\Gamma\tyu \frac{2b+n}{2}\uyt}.
\eega
which follows from \cite[Lemma C.1]{cgv2025StecconiTodino}. Specializing the above to $n=2$, we conclude.
\end{proof}
\begin{remark}
    In \cite[Theorem 3.4]{cgv2025StecconiTodino} the authors compute, in particular, $a_{b}(\|\cdot \|_{\R^n})=A(n,2b)$. Moreover, in \cite[Appendix D]{cgv2025StecconiTodino}, the following relation with the Laguerre polynomials is established:
    \be\label{eq:LaS}
L_{b}^{(\frac n2 -1)}\tyu \frac{\|x\|^2}{2}\uyt=c_{n,2b} \cdot S_{n,2b}(x)=c_{n,2b} \cdot \Sigma_{n,b}(\|x\|^2), \quad \forall x \in \R^{n}.
\ee
The constant $c_{n,2b}$ is computed in \cite[Lemma D.2]{cgv2025StecconiTodino}.
\end{remark}
In what follows, we will consider $n=2$.

\subsection{Preliminary lemmas}

\begin{definition}\label{def:2F1}
    We recall the definition and notation of the classical hypergeometric function,
    \be 
{}_{2}{F}_1\tyu a,b;c; z\uyt= \frac{\Gamma(c)}{\Gamma(a)\Gamma(b)}\sum_{k=0}^{\infty}\frac{\Gamma(a+k)\Gamma(b+k)}{\Gamma(c+k)k!} z^k,
    \ee
defined first for $|z|< 1$ and extended by analytic continuation. We will only consider it for real values of $a,b,c,z$.
\end{definition}
We recall the value of ${}_{2}{F}_1\tyu a,b;c; 1\uyt$, obtained as the limit $z\to 1-$, in the case when $c-a-b>0$, that is,
\be \label{eq:Fs0}
{}_{2}{F}_1\tyu a,b;c; 1\uyt=\frac{\Gamma\tyu c\uyt \Gamma\tyu c-a-b\uyt}{\Gamma\tyu c-a\uyt \Gamma\tyu c-b\uyt}.
\ee
We will use such formula in the proof of \cref{prop:s0} to analyze the case $s=0$ and in the proof of \cref{thm:2} below.
\begin{definition}[Coefficient $\mathcal I$]
    Define for any $i\le b$ natural numbers, and $\frac{s^2}{\xi^2}\in (0,1]$
    \be \label{eq:IBF}
\m I^b_i\tyu{\frac{s^2}{\xi^2}}\uyt:= B\left(i + \frac{1}{2},\, b - i + 1\right) \cdot
{}_2F_1\left(
b-\frac{ 1}{2},\,
i + \frac{1}{2};\,
b + \frac{3}{2};\,
1-\frac{s^2}{\xi^2}
\right),
    \ee
    where $B(z,w)=\frac{\Gamma(z)\Gamma(w)}{\Gamma(z+w)}$ is the Beta function and $\prescript{}{2}{F}_1(a,b;c;z)$ is the hypergeometric function (defined for $|z|<1$). 
\end{definition}
\begin{lemma}[Gradient's spherical chaos] \label{thm:1} The random variable $\widetilde{H}_{2b}(d_Pf)$, defined as
\be 
\widetilde{H}_{2b}(d_Pf):=
    \int_{S(T_PM)} 
 H_{2b}\tyu \frac{\langle d_Pf, u\rangle}{\|u\|_{g^f}}\uyt \|u\|_{g^f}du ,
 \ee
depends only on $\|\nabla^H_P f\|$ and $\frac1s \de_\psi f = f\tyu PR_3\tyu \frac{\pi}{2s}\uyt\uyt$, see \eqref{eq:decNabla}. In particular, it takes the following expression:
\begin{multline} 
\widetilde{H}_{2b}(d_Pf) = \xi \cdot \sum_{i+j=b}\binom{2b}{2i} \cdot \m I_{i}^b\!\tyu {\frac{s^2}{\xi^2}}\uyt\cdot \tyu {\frac{s^2}{\xi^2}}\uyt ^{i} \\
\times H_{2i}\tyu f\tyu PR_3\tyu \frac{\pi}{2s}\uyt\uyt\uyt \Lag_j\tyu \frac{\|\nabla^H_Pf\|^2}{\xi^2}\uyt.
\end{multline}
\end{lemma}
We refer to \cref{sec:auxProof} for the proof.

\begin{lemma}[Nodal volume integrated along the fibers] \label{thm:2} Let us define for any $x\in S^2$, and $a,b\in \N$
\be 
\widetilde{H}_{2a,2i}(f,x)=\int_{\pi^{-1}(x)}
 H_{2a}(f(P))
H_{2i}\tyu f\tyu 
PR_3\tyu \frac{\pi}{2s}\uyt\uyt\uyt
\dd P.
\ee
Let $P$ be some fixed element of $\pi^{-1}(x)$. 
Then, 
    \be 
\widetilde{H}_{2a,2i}(f,x)
=\frac{\Lag_{a+i}\tyu \|f\|^2_{\spin{s}_x}\uyt}{_2F_1(a,i;a+i+\frac12;1)} .
    \ee
\end{lemma}
We refer to \cref{sec:auxProof} for the proof.
\subsection{Proof of \cref{thm:3}: the case $s\neq 0$} \label{sec:proofs}
Let us give the precise definition of the coefficients $\kappa_t$, and recall the statement of our first main result. 
\begin{definition}[Coefficient $\kappa_t$]
Define \begin{multline} \label{eq:immondo}
    \immondo_t\tyu \a, \b, {\frac{s^2}{\xi^2}}\uyt:= \sum_{i=0}^{\a}
{}_2F_1\left(
\b+i-\frac{ 1}{2},\,
i + \frac{1}{2};\,
\b+i + \frac{3}{2};\,
1-\frac{s^2}{\xi^2}
\right)
\\
\times
\tyu  
{\frac{s^2}{\xi^2}}\uyt^{i} 
\nu(i,\b,\a)\frac{H_{2\a-2i}(t)}{H_{2\a-2i}(0)},
\end{multline}
where
\be 
\nu(i,\beta,\a)= \frac{B\left(i + \frac{1}{2},\, \beta + 1\right)}{_2F_1(\a-i,i;\a+\frac12;1)} \binom{2(i+\beta)}{2i}\cdot \frac{\coeff(\a-i,\beta+i)}{s_3}.
\ee
\end{definition}
\begin{manualtheorem}{3.1} 
If $s\neq 0$. For any Borel set $D\subset S^2$, we have that the $q^{th}$ chaos component of boundary area measure is given by
\begin{multline}
\frac{\lft(\pi^{-1}(D))[q]}{\xi}
=\sum_{\substack{\a,\b\in \N,\a+\b=\!\frac{q}2
\\ }}  
e^{-\frac{t^2}{2}}
\immondo_t\tyu \a,\b, \frac{s}{\xi}\uyt \\ \times
\int_D  
\Lag_{\a}(\|f\|_{\spin{s}_x}^2)
\Lag_\b\tyu \frac{\|\nabla^H_xf\|^2}{\xi^2}\uyt
\dd x.
\end{multline}
\end{manualtheorem}

\begin{remark}
    From a careful investigation of the proof, it is clear that an analogous decomposition holds for more general Borel sets $B\subset SO(3)$, that are not necessarily union of fibers of $\pi$, i.e. $B\subsetneq \pi^{-1}(\pi(B))$: \begin{multline}
        \lft(B) [q] = 
\sum_{a,b\in \N,a+b=\!\frac{q}2}
e^{-\frac{t^2}{2}}\frac{H_{2a}(t)}{H_{2a}(0)}
\!\frac{\coeff(a,b)}{s_3} \\
\times \int_{\pi(B)} \int_{\pi^{-1}(x)\cap B}
 H_{2a}(f(P)) \widetilde{H}_{2b}\tyu  d_Pf \uyt \dd P \dd x.
    \end{multline}
    See \cref{eq:decFiberWise} for comparison.
\end{remark}
\begin{proof}[Proof of \cref{thm:3}, given \cref{thm:1} and \cref{thm:2}]
    The main result of \cite{cgv2025StecconiTodino}, specialized to the three dimensional setting, is the following formula. Denoting $\|u\|_{g^f}^2=\E\kop |d_Pf(u)|^2\pok
$ for all $u\in T_xM$, then the $q^{th}$ chaos component of the nodal measure of $f$ evaluated at $B$ is 
\begin{multline}\label{eq:nodalchaosAT} 
\lft(B)[q]
=
\sum_{a,b\in \N,a+b=\!\frac{q}2}
e^{-\frac{t^2}{2}}\frac{H_{2a}(t)}{H_{2a}(0)}
\!\frac{\coeff(a,b)}{s_3} \\
\times \int_B  
 H_{2a}(f(P)) \underbrace{\int_{S(T_PM)}H_{2b}\tyu \frac{\langle d_Pf, u\rangle}{\|u\|_{g^f}}\uyt \|u\|_{g^f}du}_{=:\widetilde{H}_{2b}\tyu  d_Pf \uyt}   dP,
\end{multline}
for all $q$ even, while it is zero for all odd $q$. The constant $\coeff(a,b)$ is given by
\be\label{eq:defC} 
\coeff(a,b)=\frac{(-1)^{a+b-1}}{2^{a+b}(2b-1)a!b!},
\ee
for all $a,b\in \N$; $s_3=2\pi^2$ denotes the volume of the unit $3$-sphere, while $H_i$ denotes the Hermite polynomial of degree $i$. 

Let us consider $B=\pi^{-1}(D)\subset SO(3)$, for some $D\subset S^2$. Then, with the needed substitutions and applying \cref{thm:1}, we may write
\begin{align} 
& \lft(B) [q] \nonumber \\
& \quad = \label{eq:decFiberWise}
\sum_{a,b\in \N,a+b=\!\frac{q}2}
e^{-\frac{t^2}{2}}\frac{H_{2a}(t)}{H_{2a}(0)}
\!\frac{\coeff(a,b)}{s_3} \int_D\int_{\pi^{-1}(x)}
 H_{2a}(f(P)) \widetilde{H}_{2b}\tyu  d_Pf \uyt \dd P \dd x
 \\
 &
 \overset{\text{\cref{thm:1}}}{=}
 \sum_{a,b\in \N,a+b=\!\frac{q}2}
e^{-\frac{t^2}{2}}\frac{H_{2a}(t)}{H_{2a}(0)}
\!\frac{\coeff(a,b)}{s_3} \int_D\int_{\pi^{-1}(x)}
 H_{2a}(f(P)) 
\nonumber \\
& \qquad \times 
\xi \cdot \sum_{i+j=b} \m I_{i}^b\tyu \frac s\xi\uyt \cdot  \tyu \frac s\xi\uyt ^{2i}\binom{2b}{2i}  H_{2i}\tyu \frac{\de_\psi f}{s}\uyt \Lag_j\tyu \frac{\|\nabla^H_Pf\|^2}{\xi^2}\uyt
\ \m I_{i}^b
\dd P \dd x
\nonumber \\
 &\overset{\text{ Prop. \ref{prop:samenorm}}}{=}
 \sum_{a,b\in \N,a+b=\!\frac{q}2} \xi \cdot \sum_{i+j=b}\ \m I_{i}^b\!\tyu \frac s\xi\uyt \cdot \tyu \frac s\xi\uyt^{2i}\binom{2b}{2i}  
e^{-\frac{t^2}{2}}\frac{H_{2a}(t)}{H_{2a}(0)}
\!\frac{\coeff(a,b)}{s_3} 
\nonumber \\
&\qquad \times 
\int_D \Lag_j\tyu \frac{\|\nabla^H_xf\|^2}{\xi^2}\uyt \underbrace{\int_{\pi^{-1}(x)}
 H_{2a}(f(P))
H_{2i}\tyu \frac{\de_\psi f}{s}\uyt 
\dd P}_{=: \widetilde{H}_{2a,2i}(f,x)} \dd x; \nonumber
\end{align}
where in the last step we used that the norm of $\nabla^H f$ is invariant along the fibers, as proved in \cref{prop:samenorm}. 
Since $ \de_\psi f = s\cdot f\tyu PR_3\tyu \frac{\pi}{2s} \uyt\uyt$ (\cref{prop:inva}), applying \cref{thm:2} we can conveniently represent the term {$\widetilde{H}_{2a,2i}(f,x)$, }
defined for any $x\in S^2$, and $a,b\in \N$, in terms of $\Sigma_a(\|f\|_{\spin{s}})$, which is defined in \cref{def:spinsnorm}. 
Then, we have that 
\bega 
\lft(B)[q]&
\overset{\text{\cref{thm:2}}}{=}
\sum_{a,b\in \N,a+b=\!\frac{q}2} \xi \cdot \sum_{i+j=b}\ \m I_{i}^b\!\tyu \frac s\xi\uyt\tyu \frac s\xi\uyt^{2i}\binom{2b}{2i} 
e^{-\frac{t^2}{2}}\frac{H_{2a}(t)}{H_{2a}(0)}
\!\frac{\coeff(a,b)}{s_3} 
\\
&\qquad \times 
\int_D \Lag_j\tyu \frac{\|\nabla^H_xf\|^2}{\xi^2}\uyt 
\frac{\Lag_{a+i}\tyu \|f\|^2_{\spin{s}_x}\uyt}{_2F_1(a,i;a+i+\frac12;1)}
\dd x.
\eega
Then, taking
\bega 
\nu(i,j,a+i)
&=
\frac{B\left(i + \frac{1}{2},\, b - i + 1\right) }{_2F_1(a,i;a+i+\frac12;1)}\binom{2b}{2i}\cdot \frac{\coeff(a,b)}{s_3} \cdot 
\\
&=
\frac{B\left(i + \frac{1}{2},\, j + 1\right) }{_2F_1(a,i;a+i+\frac12;1)}\binom{2(i+j)}{2i}\cdot \frac{\coeff(a,i+j)}{s_3} \cdot 
\eega
(where $b=i+j$)
we have that
\begin{gather} 
\m I_{i}^b\!\tyu \frac s\xi\uyt\tyu \frac s\xi\uyt^{2i}\binom{2b}{2i}\cdot \frac{\coeff(a,b)}{s_3} \cdot 
\frac{1 }{_2F_1(a,i;a+i+\frac12;1)}
\\
=
{}_2F_1\left(
b-\frac{ 1}{2},\,
i + \frac{1}{2};\,
b + \frac{3}{2};\,
1-\frac{s^2}{\xi^2}
\right)\tyu \frac s\xi\uyt^{2i}\nu(i,j,a+i)=:\nu'(i,j,a+i, \frac{s}{\xi}).
\end{gather}
The previous implies that 
\begin{align}
& \frac{\lft(B)[q]}{\xi} = 
\sum_{a,b\in \N,a+b=\! \frac{q}2}  \sum_{i+j=b}
\nu'(i,j,a+i, \frac{s}{\xi})
e^{-\frac{t^2}{2}}\frac{H_{2a}(t)}{H_{2a}(0)} \\
& \hspace{5cm} \times \int_D \Lag_j\tyu \frac{\|\nabla^H_xf\|^2}{\xi^2}\uyt 
\Lag_{a+i}(\|f\|_{\spin{s}_x}^2)\dd x
\\
&=
\sum_{\substack{\a,\b\in \N,\a+\b=\!\frac{q}2
\\ }}  \sum_{i=0}^{\a}
\nu'(i,\b,\a, \frac{s}{\xi})
e^{-\frac{t^2}{2}}\frac{H_{2\a-2i}(t)}{H_{2(\a-i)}(0)}\\
& \hspace{5cm} \times \int_D \Lag_\b\tyu \frac{\|\nabla^H_xf\|^2}{\xi^2}\uyt 
\Lag_{\a}(\|f\|_{\spin{s}_x}^2)\dd x
\\
&=
\sum_{\substack{\a,\b\in \N,\a+\b=\!\frac{q}2
\\ }}  \immondo_t\tyu \a,\b, \frac{s}{\xi}\uyt
e^{-\frac{t^2}{2}}
\int_D  
\Lag_{\a}(\|f\|_{\spin{s}_x}^2)
\Lag_\b\tyu \frac{\|\nabla^H_xf\|^2}{\xi^2}\uyt
\dd x,
\end{align}
where $\a=a+i$ and $\beta=j$, so that $b=\beta+i$.
We conclude by checking the final coefficient:
\bega \label{eq:controllimmondo}
    \immondo_t\tyu \a, \b, {\frac{s^2}{\xi^2}}\uyt
    &= \sum_{i=0}^{\a}
{}_2F_1\left(
\b+i-\frac{ 1}{2},\,
i + \frac{1}{2};\,
\b+i + \frac{3}{2};\,
1-\frac{s^2}{\xi^2}
\right)\tyu {\frac{s^2}{\xi^2}}\uyt^{i}\\
& \hspace{5cm} \times \nu(i,\b,\a)\frac{H_{2\a-2i}(t)}{H_{2\a-2i}(0)}
\\
&=
\sum_{i=0}^{\a}
\nu'(i,\b,\a, \frac{s}{\xi})\frac{H_{2\a-2i}(t)}{H_{2\a-2i}(0)}.
\eega

\end{proof}
 
\subsection{Proof of \cref{prop:s0}: the case $s=0$}
In the case $s=0$, the spin condition \cref{eq:spinproperty}, entails that $f$ is constant on the fibers of $\pi\colon SO(3)\to S^2$, that is $f=\phi\circ \pi$ for a Gaussian function $\phi$ on $S^2$. Moreover, in such case, the left-invariance \cref{prop:inva} of $f$ translates into $\phi$ being an isotropic Gaussian field on $S^2$, that is, invariant in law under rotations. What happens to the level set in this case is clear:
$
f^{-1}(t)=\pi^{-1}\tyu \phi^{-1}(t)\uyt
$ is a union of fibers.
Thus, given that each fiber has length $2\pi$ and that $\pi$ is a Riemannian submersion, we can express the Area of a level of $f$ as the length of the level of $\phi$ via the identity
\be 
\m L_{f-t}(\pi^{-1}(D))=2\pi\cdot \m L_{\phi-t}(D).
\ee
One can thus obtain directly the following chaos expansion.
\begin{proof}[Proof of \cref{prop:s0}]
    A direct consequence of \cref{eq:nodalchaosAT} (from \cite{cgv2025StecconiTodino}) and \cref{def:sigma} is the following formula
    \begin{multline} 
\frac{\m L_{f-t}(\pi^{-1}(D))}{\xi}[q] \\ =
 \sum_{a,b\in \N,a+b
  =\!\frac{q}2}
e^{-\frac{t^2}{2}}\frac{H_{2a}(t)}{H_{2a}(0)}
\!\frac{\coeff(a,b)}{s_2} \int_D
2\pi\cdot
 H_{2a}(\phi(x))
 \Sigma_b\tyu \frac{\|\nabla_x \phi\|^2}{\xi^2}\uyt 
 dx.
    \end{multline}
    Now, recalling the definitions of the coefficient, we may write
    \bega 
\tyu \frac{H_{2a}(t)}{H_{2a}(0)} \!
\frac{\coeff(a,b)}{s_2}\uyt^{-1} \! \immondo_t(a,b,0)
&=\frac{s_2}{s_3} {_2F_1}\left( b-\frac{ 1}{2},\, \frac{1}{2};\, b+ \frac{3}{2};\, 1 \right)
\frac{B\left(\frac12,b+1\right) }{_2F_1(\a,0;\a+\frac12;1) }\\
&= \frac{4\pi}{2\pi^2} 
\frac{\Gamma\tyu b+\frac32\uyt\Gamma\tyu \frac32\uyt}{\Gamma\tyu 2\uyt\Gamma\tyu b+1\uyt}
\frac{
\frac{\Gamma\tyu\frac12\uyt \Gamma\tyu b+1\uyt}{\Gamma\tyu b+\frac32\uyt}
}
{
\frac{\Gamma\tyu \a+\frac12\uyt \Gamma\tyu \frac12\uyt}{\Gamma\tyu \frac12\uyt \Gamma\tyu \a+\frac12\uyt}
}
\\
&=
\frac{4\pi}{2\pi^2} \frac{\pi}{2}=1.
    \eega
The second identity follows from well known properties of the hypergeometric function, namely, its evaluation at $z=1$. This concludes the proof.
\end{proof}

\subsection{Proof of preliminary lemmas}\label{sec:auxProof}

\begin{proof}[Proof of \cref{thm:1}] Recall \cref{def:verHor} and \cref{prop:ortogonality}. 
Let us use the change of coordinates: $u=\sin \a \cdot h+\cos\a \cdot V(P)$, for $h\in S(\Ho(P))$
\bega 
\widetilde{H}_{2b}(d_Pf)=&
    \int_{S(T_PM)} 
 H_{2b}\tyu \frac{\langle d_Pf, u\rangle}{\|u\|_{g^f}}\uyt \|u\|_{g^f}du 
 \\
 =&
    \int_{0}^\pi\sin(\a)\int_{S(H_P)} 
 H_{2b}\tyu \frac{\sin\a \langle d_Pf, h\rangle+\cos\a \langle d_Pf, V(P)\rangle}{\sqrt{\xi^2\sin^2 \a+s^2\cos^2 \a}}\uyt \\
 & \hspace{5cm} \times \sqrt{\xi^2\sin^2 \a+s^2\cos^2 \a}\dd h \dd \a 
 \eega
 Let us introduce some notations.
 \be
s(\a):=\frac{\xi \sin\a}{\sqrt{\xi^2\sin^2 \a+s^2\cos^2 \a}}, \quad c(\a):=\frac{s\cos\a }{\sqrt{\xi^2\sin^2 \a+s^2\cos^2 \a}},
 \ee
 \be 
\nu(\a):=\sqrt{\xi^2-(\xi^2-s^2)\cos^2\a},
 \ee
 and
 \be 
\gamma_H(h):=\frac1\xi \langle d_Pf, h\rangle,
\quad 
\gamma_V:=\frac1s \langle d_Pf, V(P)\rangle.
 \ee
In particular, we have that \be \label{eq:dNabla}
\gamma_H(h)=\left\langle \frac{\nabla^H_Pf}{\xi} , h\right\rangle, \quad   \gamma_V=f\tyu PR_3\tyu \frac{\pi}{2s}\uyt\uyt.
\ee
From \cref{prop:ortogonality}, we deduce that $\gamma_H(h)$ and $\gamma_V$ are independent $\m N(0,1)$ variables, for any $h\in S(\Ho(P))$.
Substituting the new notations and applying Lemma~\ref{lem:sumprod}, we can proceed the computation as follows
{\scriptsize \bega 
\widetilde{H}_{2b}(d_Pf)
&=
\int_{0}^\pi\sin\a\int_{S(H_P)} 
 H_{2b}\tyu c(\a)\gamma_V+s(\a)\gamma_H(h)\uyt \nu(\a)\ \dd h \dd \a 
 \\
& =
\sum_{i+j=2b}\binom{2b}{i}  H_i\tyu \gamma_V\uyt  \tyu\int_{S(H_P)} 
 H_j\tyu \gamma_H(h)\uyt  \ \dd h \uyt\int_{0}^\pi\nu(\a)\sin(\a)
 c(\a)^is(\a)^j\dd \a\\
 & =
\sum_{i+j=b}s^{2i}\xi^{2j}\binom{2b}{2i}  H_{2i}\tyu \gamma_V\uyt  \tyu\int_{S(H_P)} 
  H_{2j}\tyu \gamma_H(h)\uyt  \ \dd h \uyt 
  \int_{0}^\pi\frac{
  \cos(\a)^{2i}\sin(\a)^{2j+1}}{\nu(\a)^{2b-1}}\dd \a, 
\eega} %
where the last equality follows from the fact that $j$ has to be even, for the corresponding term to be non-zero. Now, define the constant \be 
\m I_{i}^b\tyu\frac{s^2}{\xi^2}\uyt =\int_{0}^\pi\frac{
 \cos(\a)^{2i}\sin(\a)^{2(b-i)+1}}{\tyu1-\epsilon\cos^2\a\uyt^{\frac{2b-1}2}}\dd \a, \qquad  \text{for } \epsilon = (1-\frac{s^2}{\xi^2}),
\ee
which is computed in \cref{lem:Iib}. Then, recalling \cref{eq:dNabla}, we conclude that
{\scriptsize \bega
\widetilde{H}_{2b}(d_Pf)
 &=
\sum_{i+j=b}s^{2i}\xi^{2j}\binom{2b}{2i}  H_{2i}\tyu \gamma_V\uyt \Lag_j(\|\gamma_H\|^2)
\ \m I_{i}^b\frac{1}{\xi^{2b-1}}
\\
 & =
\xi \cdot \sum_{i+j=b}\tyu \frac s\xi\uyt ^{2i}\binom{2b}{2i}  H_{2i}\tyu \gamma_V\uyt \Lag_j\tyu \frac{\|\nabla^H_Pf\|^2}{\xi^2}\uyt
\ \m I_{i}^b\\
& = \xi \cdot \sum_{i+j=b}\binom{2b}{2i}\cdot\m I_{i}^b\!\tyu {\frac{s^2}{\xi^2}}\uyt\cdot \tyu {\frac{s^2}{\xi^2}}\uyt ^{i} \cdot  H_{2i}\tyu f\tyu PR_3\tyu \frac{\pi}{2s}\uyt\uyt\uyt \Lag_j\tyu \frac{\|\nabla^H_Pf\|^2}{\xi^2}\uyt.
\eega} %
\end{proof}
\begin{lemma}\label{lem:Iib}
    \be
\m I_{i}^b\tyu\frac{s^2}{\xi^2}\uyt
= B\left(i + \frac{1}{2},\, b - i + 1\right) \cdot
{}_2F_1\left(
b-\frac{ 1}{2},\,
i + \frac{1}{2};\,
b + \frac{3}{2};\,
1-\frac{s^2}{\xi^2}
\right)
\ee

\end{lemma}
\begin{proof}
Changing variable $x=\cos\a$, and then $x^2=t$, the integral becomes
\be 
\m I_{i}^b=\int_{-1}^{1} \frac{x^{2i} (1 - x^2)^{b - i}}{\left(1 - \epsilon x^2\right)^{\frac{2b - 1}{2}}} \, dx
=
\int_{0}^{1} \frac{t^{i-\frac12} (1 - t)^{b - i}}{\left(1 - \epsilon t\right)^{\frac{2b - 1}{2}}} \, dt.
\ee
The latter fits in the classic formulation of Gauss' hypergeometric integral (see \cite[Equations (15.1.1), (15.3.1)]{Abramowitz}):
\be 
\int_{0}^1 t^{p-1} (1-t)^{q-1}(1-zt)^{-r}dt=B(p,q)\cdot {}_2{F}_1\tyu r,p ; p+q ; z\uyt,
\ee
where $B(p,q)=\frac{\Gamma(p)\Gamma(q)}{\Gamma(p+q)}$ denotes the Beta function, allowing us to conclude.
\end{proof}
\begin{proof}[Proof of \cref{thm:2}] To avoid ambiguities between the index of the Hermite polynomials, and the imaginary unit, we prove the equality for $\widetilde H_{2a,2m}(f,x)$.
\begin{align} 
\widetilde{H}_{2a,2m}(f,x) &=
\int_{0}^{2\pi}
 H_{2a}(f(P(x)R_3(\psi)))
H_{2m}\tyu f\tyu 
P(x)R_3\tyu \psi+ \frac{\pi}{2s}\uyt\uyt\uyt
\dd \psi
\\
&=
\int_{0}^{2\pi}
 H_{2a}\tyu \Re ( X e^{is\psi})\uyt 
H_{2m}\tyu \Re ( Xie^{is\psi})\uyt
\dd \psi
\\
&=
\int_{0}^{2\pi}
 H_{2a}\tyu \Re ( X e^{i\psi})\uyt 
H_{2m}\tyu \Re ( Xie^{i\psi})\uyt
\dd \psi,
\end{align}
(for the last identity observe that the integrand is $\frac{2\pi}{s}$-periodic and use the change of variables $\psi\mapsto s\psi$).
This random variable is in the chaos of order $2m+2a$, with respect to $X$ and it is invariant under rotations of $X$, so it must depend only on $|X|$. Moreover, since the even Hermite polynomials contain only even degree terms, we deduce that the above is a polynomial in the variable $|X|^2$. Following on this argument (see also \cref{prop:ortogchi}), we can conclude that there exists a constant $c_{a,m}$ such that
\be 
\widetilde{H}_{2a,2m}(f,x) = c_{a,m}\int_{0}^{2\pi} H_{2a+2m}\tyu \langle X,u \rangle_{\R^2} \uyt \dd u,
\ee
which we can write in complex notation as
\begin{align}
\widetilde{H}_{2a,2m}&(f,x) 
 =c_{a,m}\int_{0}^{2\pi} H_{2a+2m}\tyu \Re\tyu X e^{i\psi}\uyt\uyt d\psi
\\
& =
c_{a,m}\int_{0}^{2\pi} H_{2a+2m}\tyu f(P(x))\cos\psi +f\tyu P(x)R_3\tyu\frac{\pi}{2s}\uyt\uyt \sin(\psi)\uyt d\psi.
\end{align}
The constant $c_{a,m}$ is computed in \cref{lem:cam}, allowing us to conclude.
\end{proof}
\begin{lemma}\label{lem:cam}
    \be 
c_{a,m}=
 \frac{\Gamma\tyu a+\frac12\uyt\Gamma\tyu m+\frac12\uyt}{\Gamma\tyu a+m+\frac12\uyt\sqrt{\pi}}
 =\ \frac{1}{_2F_1(a,m;a+m+\frac12;1)}
    \ee
\end{lemma}
\begin{proof} Recall that
{\small \begin{align}
& c_{a,m}=\frac{E_1}{E_2}\\
& =\frac{
\E\kop \int_{0}^{2\pi}
 H_{2a}\tyu \Re ( X e^{i\psi})\uyt 
H_{2m}\tyu \Re ( Xie^{i\psi})\uyt
\dd \psi 
\cdot \int_{0}^{2\pi} H_{2a+2m}\tyu \Re\tyu X e^{i\psi'}\uyt\uyt d\psi'\pok
} 
{
\Var \kop \int_{0}^{2\pi} H_{2a+2m}\tyu \Re\tyu X e^{i\psi'}\uyt\uyt d\psi'\pok}.
\end{align}} %
First, we analyze the numerator:
{\small \begin{align} 
& E_1 = \\
&=\E\kop \int_{0}^{2\pi}
 H_{2a}\tyu \Re ( X e^{i\psi})\uyt 
H_{2m}\tyu \Re ( Xie^{i\psi})\uyt
\dd \psi 
\cdot \! \int_{0}^{2\pi} H_{2a+2m}\tyu \Re\tyu X e^{i\psi'}\uyt\uyt d\psi'\pok 
\\
&=
\int_{0}^{2\pi}\dd \psi\int_{0}^{2\pi}\dd \psi'\E\kop 
 H_{2a}\tyu \Re ( X e^{i\psi})\uyt 
H_{2m}\tyu \Re ( Xie^{i\psi})\uyt
 H_{2a+2m}\tyu \Re\tyu X e^{i\psi'}\uyt\uyt\pok
\\
&=
2\pi\int_{0}^{2\pi}\dd \psi\E\kop 
 H_{2a}\tyu \Re ( X)\uyt 
H_{2m}\tyu \Re ( Xi)\uyt
 H_{2a+2m}\tyu \Re\tyu X e^{i\psi}\uyt\uyt\pok
 \\
&=
2\pi\int_{0}^{2\pi}\E\kop 
 H_{2a}\tyu \gamma\uyt 
H_{2m}\tyu \gamma'\uyt
 H_{2a+2m}\tyu  \gamma  \cos\psi+\gamma' \sin\psi \uyt\pok
 \dd \psi\overset{\cref{lem:sumprod}}{=}
 \\
 & = 
 2\pi\int_{0}^{2\pi}\E\kop 
 H_{2a}\tyu \gamma\uyt 
H_{2m}\tyu \gamma'\uyt
 \binom{2a+2m}{2m} \cos^{2a}\psi \sin^{2m}\psi \cdot H_{2a}(\gamma)
H_{2m}\tyu \gamma' \uyt
 \pok
 \dd \psi
 \\
&=
2\pi (2a+2m)!\int_{0}^{2\pi}\cos^{2a}\psi \sin^{2m}\psi 
 \dd \psi,
\end{align} } %
where $\gamma,\gamma'\sim \m N(0,1)$ are independent. In the third identity we use the fact that $X$ and $Xe^{i\psi}$ have the same law, and then redefine $\psi$ as $\psi'-\psi$.
Now, we analyze the denominator:
\bega 
E_2&=
\int_{0}^{2\pi}\dd \psi\int_{0}^{2\pi}\dd \psi'\E\kop 
 H_{2a+2m}\tyu \Re\tyu X e^{i\psi}\uyt\uyt
 \cdot
 H_{2a+2m}\tyu \Re\tyu X e^{i\psi'}\uyt\uyt\pok
 \\
 &=
 2\pi\int_{0}^{2\pi}\dd \psi\E\kop 
 H_{2a+2m}\tyu \Re\tyu X\uyt\uyt
 \cdot
 H_{2a+2m}\tyu \Re\tyu X e^{i\psi}\uyt\uyt\pok
 \\
 &=
2\pi(2a+2m)!\int_{0}^{2\pi}(\cos\psi)^{2a+2m}\dd \psi.
\eega
To conclude,
\be 
c_{a,m}=\frac{\int_{0}^{2\pi}\cos^{2a}\psi \sin^{2m}\psi 
 \dd \psi}{\int_{0}^{2\pi}(\cos\psi)^{2a+2m}\dd \psi}=
 \frac{B\tyu a+\frac12, m+\frac12 \uyt}{B(a+m+\frac12,\frac12)}
 =
 \frac{\Gamma\tyu a+\frac12\uyt\Gamma\tyu m+\frac12\uyt}{\Gamma\tyu a+m+\frac12\uyt\sqrt{\pi}}.
    \ee
    We conclude using \cref{eq:Fs0}.
\end{proof}

\begin{appendix} 
\section{Gaussian Formulas}\label{appendix}
\begin{lemma}\label{lem:sumprod}
Let $c^2+s^2=1$ and $x,y\in\R$. Then,
    \be 
H_{q}\tyu cx+sy \uyt=\sum_{i+j=q}H_i(x)H_j(y)\binom{q}{i} c^is^j.
\ee
\end{lemma}
\begin{proof}
We present an three-line proof based on the generating function, see \cref{eq:H}. 
\bega
\sum_{q\in \N}H_q\tyu cx+sy \uyt\frac{t^q}{q!} &=e^{t\tyu cx+sy \uyt-\frac{t^2}{2}}
=
e^{t\tyu cx+sy \uyt-\frac{t^2(c^2+s^2)}{2}}
\\
&=
e^{(tc)x-\frac{(tc)^2}{2}}
\cdot
e^{(ts)y-\frac{(ts)^2}{2}}
=
\sum_{i,j\in \N}H_i(x)H_j(y)\frac{(ct)^i}{i!}\frac{(st)^j}{j!}
\\
&=\sum_{q\in\N}\sum_{i+j=q}H_i(x)H_j(y)\frac{c^i}{i!}\frac{s^j}{j!} t^q.
\eega
\end{proof}

The following result reformulates \cref{thm:1} in full generality. 

\begin{theorem} Let us consider a Gaussian vector $\eta \sim \mathcal N(0,C)$, where \begin{equation}
    C=\begin{pmatrix}
        \xi^2 \1_{n-1} & 0 \\ 0 & s^2
    \end{pmatrix}.
\end{equation} 
Fix $b\in\N$. The random variable $F(\eta)$, defined as
\be 
F(\eta):=
    \int_{S^{n-1}} 
 H_{2b}\tyu \frac{\langle \eta, u\rangle}{\|C^{1/2} u\| }\uyt \|C^{1/2} u\| du ,
 \ee
depends only on $\|\eta_H\|$ and $\eta_V$, where $\eta_H = (\eta_1,\ldots,\eta_{n-1})\sim\mathcal N(0,\xi^2 \1_{n-1})$, and $\eta_V=\eta_n\sim\mathcal N(0,s^2)$. In particular, it takes the following expression:
\be 
F(\eta)=\xi \cdot \sum_{i+j=b}\binom{2b}{2i}\cdot\m I_{i}^b\!\tyu {\frac{s^2}{\xi^2}}\uyt\cdot \tyu {\frac{s^2}{\xi^2}}\uyt ^{i} \cdot  H_{2i}\tyu \frac{\eta_V}{s}\uyt \Lag_{n-1,j}\tyu \frac{\|\eta_H\|^2}{\xi^2}\uyt,
\ee
where $\Sigma_j$ are as in \eqref{eq:lag}.
\end{theorem}

The following result reformulates \cref{thm:2} in full generality, using three different formulations.

\begin{theorem}

\textbf{(As a complex rv)} Let $\zeta \sim \mathcal N_{\mathbb C}(0,2)$ be a (symmetric) complex Gaussian variable such that $\Re(\zeta)$, $\Im(\zeta)$ are independent standard (real) Gaussian variable (see \cite[Proposition 1.31]{janson}). Fix $a,b\in\N$. Then, \begin{equation}\label{eq:complexForm}
    \int_{0}^{2\pi}
 H_{2a}\tyu \Re ( \zeta e^{i\psi})\uyt 
H_{2b}\tyu \Re ( \zeta ie^{i\psi})\uyt
\dd \psi 
=\frac{\Lag_{a+b}\tyu 
|\zeta |^2 \uyt}{_2F_1(a,b;a+b+\frac12;1)}
,
\end{equation}
where $|z|=\sqrt{z\bar z}$ denotes the norm of a complex number $z$, ${}_2F_1$ is the hypergeometric function, see \cref{def:2F1}, and $\Sigma$ as in \eqref{eq:lag}.

\textbf{(As a random vector)} An equivalent formulation comes from identifying $\C \cong \R^2$. Let $\eta = \binom{\gamma_1 \cos\varphi}{\gamma_2\sin\varphi}$ be a Gaussian vector, where $\gamma_1$, $\gamma_2\sim\mathcal N_\R(0,1)$ are independent. Let $R
$ be a rotation of $\frac\pi2$. Then, $R\eta = \binom{-\gamma_2\sin\varphi}{\gamma_1 \cos\varphi}$, and \begin{equation}\label{eq:vectorForm}
    \int_{S^1}
 H_{2a}\tyu \langle \eta, v \rangle \uyt 
H_{2b}\tyu  \langle R\eta, v \rangle \uyt
\dd v 
=\frac{\Lag_{a+b}\tyu 
|\eta |^2 \uyt}{_2F_1(a,b;a+b+\frac12;1)},
\end{equation}
where $|z|=\sqrt{z_1^2 + z_2^2}$ denotes the norm of a vector $z\in\R^2$, ${}_2F_1$ is the hypergeometric function, see \cref{def:2F1}, and $\Sigma$ as in \eqref{eq:lag}.

\textbf{(Explicit)} A third equivalent formulation is the following. Let us consider independent $\gamma_1$, $\gamma_2\sim\mathcal N(0,1)$. Then, \begin{multline}\label{eq:explicitForm}
    \int_{0}^{2\pi}
 H_{2a}\tyu \gamma_1 \cos\psi - \gamma_2 \sin\psi \uyt 
H_{2b}\tyu \gamma_1 \sin\psi + \gamma_2 \cos\psi \uyt
\dd \psi \\
=\frac{\Lag_{a+b}\tyu 
\gamma_1^2 + \gamma_2^2  \uyt}{_2F_1(a,b;a+b+\frac12;1)}
,
\end{multline}
where ${}_2F_1$ is the hypergeometric function, see \cref{def:2F1}, and $\Sigma$ as in \eqref{eq:lag}. 
\end{theorem}

\end{appendix}

\subsubsection*{Acknowledgments.} The authors acknowledge support from Workshop INdAM - 2024 ``Analysis and Geometry of Random Fields'' (CUP E53C23000950001). Michele Stecconi acknowledges support from the Luxembourg National Research Fund (Grant: 021/16236290/HDSA). FP’s research is partially supported by the Luxembourg National Research Fund (Grant: O22/17372844/FraMStA).
\bibliographystyle{abbrv}
\bibliography{spin_chaos.bib}
\end{document}